\RequirePackage[l2tabu, orthodox]{nag}

\documentclass[12pt]{amsart}
\usepackage{fullpage,url,amssymb,enumitem,colonequals}
\makeatletter
\newcommand{\borysitem}[1]{%
	\item[#1]\protected@edef\@currentlabel{#1}%
}
\makeatother
\usepackage[all]{xy} 
\usepackage{mathrsfs} 

\usepackage[OT2,T1]{fontenc}
\DeclareSymbolFont{cyrletters}{OT2}{wncyr}{m}{n}
\DeclareMathSymbol{\Sha}{\mathalpha}{cyrletters}{"58}

\usepackage{xcolor}

\newcommand{\defi}[1]{\textsf{#1}} 

\newcommand{\F}{\mathbb{F}}

\newcommand{\N}{\mathbb{N}}
\newcommand{\PP}{\mathbb{P}}
\newcommand{\Q}{\mathbb{Q}}

\newcommand{\Z}{\mathbb{Z}}

\newcommand{\calA}{\mathcal{A}}
\newcommand{\calB}{\mathcal{B}}

\DeclareMathOperator{\Aut}{Aut}

\DeclareMathOperator{\Gal}{Gal}

\DeclareMathOperator{\Hom}{Hom}

\DeclareMathOperator{\im}{im}

\newcommand{\slantsf}[1]{\textsl{\textsf{#1}}}

\newtheorem{theorem}{Theorem}[section]
\newtheorem{lemma}[theorem]{Lemma}

\theoremstyle{definition}
\newtheorem{definition}[theorem]{Definition}

\newtheorem{conjecture}[theorem]{Conjecture}

\theoremstyle{remark}
\newtheorem{remark}[theorem]{Remark}

\usepackage{microtype}

\usepackage[
	backref,
	pdfauthor={Philip Dittmann and Borys Kadets}, 
]{hyperref}
\usepackage[alphabetic,backrefs,lite]{amsrefs} 

\begin{document}

\title[Odoni's conjecture is false]{Odoni's conjecture on arboreal Galois representations is false}
\author{Philip Dittmann}
\address{Philip Dittmann, Technische Universität Dresden, Fakultät Mathematik, Institut für Algebra, 01062 Dresden, Germany}
\email{philip.dittmann@tu-dresden.de}
\author{Borys Kadets}
\address{Borys Kadets, Mathematical Sciences Research Institute, 17 Gauss Way, Berkeley, CA 94720-5070}
\thanks{This material is based upon work supported by the National Science Foundation under Grant No. DMS-1928930 while the authors participated in a program hosted by the Mathematical Sciences Research Institute in Berkeley, California, during the Fall 2020 semester.}
\email{kadets.math@gmail.com}
\urladdr{\url{http://bkadets.github.io}}

\begin{abstract}
Suppose $f \in K[x]$ is a polynomial. The absolute Galois group of $K$ acts on the preimage tree $\mathrm{T}$ of $0$ under $f$. The resulting homomorphism $\phi_f\colon \Gal_K \to \Aut \mathrm{T}$ is called the arboreal Galois representation. Odoni conjectured that for all Hilbertian fields $K$ there exists a polynomial $f$ for which $\phi_f$ is surjective. We show that this conjecture is false.
\end{abstract}

\maketitle

\section{Introduction}\label{S:introduction}

Suppose that $K$ is a field and $f \in K[x]$ is a polynomial of degree $d$. Suppose additionally that $f$ and all of its iterates $f^{\circ k}(x)\colonequals f\circ f \circ \dots \circ f$ are separable. To $f$ we can associate the arboreal Galois representation -- a natural dynamical analogue of the Tate module -- as follows. Define a graph structure on the set of vertices $V\colonequals \bigsqcup_{k\geqslant 0} \left(f^{\circ k}\right)^{-1}(0)$ by drawing an edge from $\alpha$ to $\beta$ whenever $f(\alpha)=\beta$. The resulting graph is a complete rooted $d$-ary tree $\mathrm{T}_\infty(d)$. The Galois group $\Gal_K$ acts on the roots of the polynomials $f^{\circ k}$ and preserves the tree structure; this defines a morphism $\phi_f\colon \Gal_K \to \Aut \mathrm{T}_\infty(d)$ known as the \defi{arboreal representation} attached to $f$.
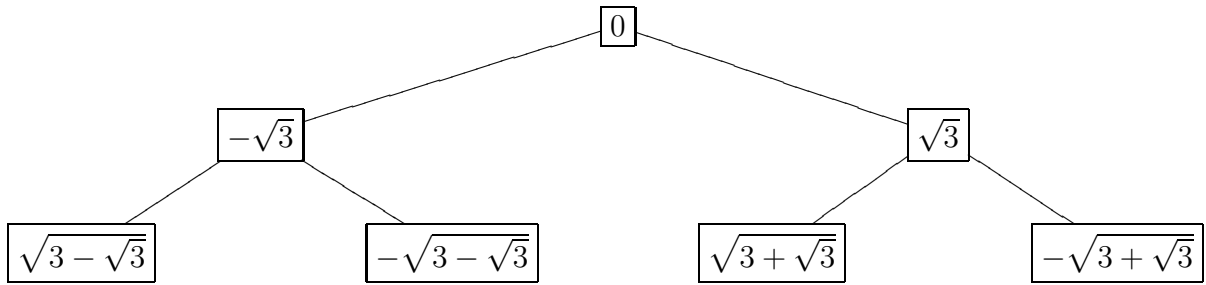
\begin{figure}[h]
\centerline{
	\xymatrix{
		&  & &*+[F]{0} \ar@{-}[drr] \ar@{-}[dll] & && \\
		&*+[F]{ -\sqrt{3}} \ar@{-}[dr] \ar@{-}[dl] & & & &*+[F]{\sqrt{3}} \ar@{-}[dr] \ar@{-}[dl] & \\
		*+[F]{\sqrt{3-\sqrt{3}}} &  &*+[F]{-\sqrt{3-\sqrt{3}}} & &*+[F]{\sqrt{3+\sqrt{3}}} & &*+[F]{-\sqrt{3+\sqrt{3}}} 
}
}
\caption{First two levels of the tree $\mathrm{T}_\infty(2)$ associated with the polynomial $f=x^2-3$}
\end{figure}

This definition is analogous to that of the Tate module of an elliptic curve, where the polynomial $f$ is replaced by the multiplication-by-$p$ morphism. However, in contrast to the case of Tate modules, arboreal representations are rather poorly understood. In particular, the images of arboreal representations are often expected to be large (see \cite{Jones2013} for a survey), even though describing these images is out of reach even for some quadratic polynomials. 

In \cite{Odoni1985a} Odoni proves that for the generic polynomial the associated arboreal representation is surjective. Therefore, over fields in which Hilbert's irreducibility theorem holds -- the so-called Hilbertian fields -- for every integer $k$ the Galois group of $f^{\circ k}$ is maximal for infinitely many polynomials $f$. He then asks if the same holds for the whole arboreal representation.

\begin{conjecture}[\cite{Odoni1985a}*{Conjecture~7.5.}]\label{Odoni}
	Suppose $K$ is a Hilbertian field of characteristic zero, and $d>1$ is an integer. Then there exists a degree $d$ monic polynomial $f \in K[x]$ such that the associated arboreal representation $\phi_f\colon \Gal_k \to \Aut \mathrm{T}_\infty(d)$ is surjective.
\end{conjecture}

Most of the work on arboreal representations focuses on the cases when $K$ is a number field or a function field. For example, Odoni's conjecture is known to be true for all number fields; this was proved in varying degrees of generality in \cite{Specter2018-preprint}, \cite{Benedetto-Juul2019}, \cite{Kadets2020}. The goal of this paper is to disprove Odoni's conjecture.

\begin{theorem}\label{Odonifail}
	Suppose $k$ is a countable Hilbertian field of characteristic zero. There exists a Hilbertian algebraic field extension $F/k$ such that for every $f \in F[x]$ of degree $d \geqslant 2$ the image of the associated arboreal representation has infinite index in $\Aut \mathrm{T}_\infty(d)$.
\end{theorem}

We give two separate proofs for Theorem \ref{Odonifail}.
The first proof involves a very explicit general construction applicable to countable collections of special field extensions, of which finite index arboreal representations are a special case. To state the more general result we need the following definition.

\begin{definition}
	An algebraic field extension $L/K$ is called \slantsf{vast} if $L \neq K$ and for every finite extension $F/K$ there exists a subfield $M \supsetneq K$ of $L$ such that $M$ and $F$ are linearly disjoint over $K$.
\end{definition} 

The general result that will be used to prove Theorem \ref{Odonifail} is the following theorem.
\begin{theorem}\label{MainTheorem}
	Let $k$ denote a countable Hilbertian field of characteristic $0$. Suppose $\calA$ is a countable collection of vast extensions $L/K$, such that for every $L/K \in \calA$ the field $K$ is a finite extension of $k$. Then there exists an algebraic Hilbertian field extension $F/k$ such that:
	\begin{enumerate}
		\item the degree $[F:k]$ divides the product of degrees of the extensions from $\calA$ (as supernatural numbers);
		\item for every subextension $K/k$ of $F/k$ and every extension $L/K$ from $\calA$ the extensions $F/K$ and $L/K$ are not linearly disjoint.
	\end{enumerate}
\end{theorem}
From this theorem it is easy to deduce Theorem \ref{Odonifail}, with a field $F/k$ whose degree (as a supernatural number) is a power of two.

In Section \ref{sec:modelTheory}, we give a second proof of Theorem \ref{Odonifail} using model-theoretic techniques.
The counterexamples to Odoni's conjecture constructed there are of a very special kind, as they are pseudo-algebraically closed.
The argument crucially relies on the fact that the class of Hilbertian pseudo-algebraically closed fields is model-theoretically well understood.
However, the result obtained in this way through an abstract existence theorem is slightly weaker than what is given by the first proof, see Remark \ref{rem:firstProofImpliesSecond}.
In particular, weaker control on the degree of $F/k$ is obtained in the second proof, see Remark \ref{rem:comparison}; for instance when $k=\Q$ the degree of $F$ over $\Q$ obtained in this way is divisible by every natural number.
Nevertheless the second proof seems to be of independent interest due to the techniques used.

\section{An explicit construction of a counterexample}
\label{sec:explicitConstruction}

Throughout the paper we assume that all fields have characteristic zero. We begin by recalling the definition of a Hilbertian field.

\begin{definition}\label{Hilbertian}
	A characteristic zero field $K$ is called \slantsf{Hilbertian} if for every smooth geometrically integral curve $X/K$ and any nonconstant morphism $f\colon X \to \PP^1_K$ of degree at least $2$, there are infinitely many points $x \in \PP^1(K)$ such that $f^{-1}(x)$ is integral.
\end{definition}
Over such fields Hilbert's irreducibility theorem holds; hence the name.
See \cite[Chapters 12, 13]{Fried-Jarden2008} for a detailed discussion of Hilbertian fields.
\begin{remark}
  There are a few (equivalent) definitions of a Hilbertian field in the literature; the equivalence of Definition \ref{Hilbertian} with other commonly used definitions is proved in \cite[Theorem 1.1]{Bary-Soroker2008}.
  Note that \cite[Theorem 1.1]{Bary-Soroker2008} writes the main results in terms of polynomials. We now explain why the two definitions are equivalent. To disambiguate the notational conflicts with \cite{Bary-Soroker2008}, we write the notation from \cite{Bary-Soroker2008} in bold. To arrive at Definition \ref{Hilbertian} in the notation of \cite{Bary-Soroker2008} take the curve $X$ of Definition \ref{Hilbertian} to be the normalization of the projective closure of $\bf f(T,X)=0$ in \cite{Bary-Soroker2008} and the morphism $f$ of Definition \ref{Hilbertian} to be be the $\bf T$ coordinate on $\bf f(T,X)=0$. Definition \ref{Hilbertian} is not more general then the one from \cite{Bary-Soroker2008}, since every $X, f$ from definition \ref{Hilbertian} has a birational model of the form $\bf f(T,X)=0$ with the map $f$ given by the $\bf T$-coordinate map.
\end{remark}
Before proving Theorem \ref{MainTheorem}, we need the following simple property of vast extensions.
\begin{lemma}\label{vastbasechange}
	Let $L/K$ be a vast extension. Suppose $K'/K$ is a finite extension linearly disjoint from $L/K$. Then $LK'/K'$ is vast. 
\end{lemma}
\begin{proof}
Suppose $F/K'$ is a finite extension. We need to construct a finite subextension of $LK'/K'$ linearly disjoint from $F$. Consider the extension $F/K$. Since $L/K$ is vast, there exists a finite subextension $M/K$ of $L/K$ linearly disjoint from $F/K$. The compositum $FM/K$ of $F$ and $M$ has degree $[F:K][M:K]$. Since $K' \subset F \subset FM$, the extension $M/K$ is linearly disjoint from $K'/K$ and the degree of $FM$ over $K'$ is \[[FM:K']=[F:K][M:K][K':K]^{-1}=[F:K'][M:K]=[F:K'][K'M:K'].\]
Since $FM/K'$ is the compositum of $F/K'$ and $K'M/K'$, the degree formula above implies that $F/K'$ and $K'M/K'$ are linearly disjoint. Thus $K'M/K'$ is a subfield of $LK'/K'$ linearly disjoint from $F/K'$.
\end{proof}

\begin{proof}[Proof of Theorem \ref{MainTheorem}]
	By Lemma \ref{vastbasechange}, without loss of generality, we can assume that if $L/K$ is an element of $\calA$ and $K'/K$ is a finite extension linearly disjoint from $L$, then $LK'/K'$ is also an element of $\calA$. We can also assume that no finite extension $F/k$ satisfies the conclusion of the theorem.
	
	The idea of the proof is to carefully construct a tower of extensions $k=F_1 \subset F_2 \subset \dots$ such that all elements of $\calA$ defined over a finite subextension of $F = \bigcup F_n$ are not linearly disjoint from $F$, and yet $F$ is Hilbertian.
        This is achieved by a  ``diagonal argument'': we inductively produce extensions $F_{n+1}/F_n$ that are not linearly disjoint from a smallest element of $\calA$ in a certain ordering, while at the same time keeping fibers of finitely many coverings integral when base changed to $F_n$ to eventually force $F$ to be Hilbertian.
 	
	Fix an ordering of all elements of $\calA$ by natural numbers. Let $\calB$ be a list of all covers of smooth geometrically integral curves $f\colon X \to \PP^1,\ \deg f \geqslant 2$ defined over some finite extension of $k$, with every covering repeated in $\calB$ infinitely many times. We now produce sequences of fields $F_1\subset F_2 \subset \dots \subset F_n \subset \dots $, coverings $f_n \in \calB, f_n\colon X_n \to \PP^1$, points $c_n \in \PP^1$, and extensions $L_n/K_n \in \calA$, indexed by natural numbers, with the following properties, that depend on a parameter $m \in \N$.
	\begin{enumerate}
		\borysitem{$(1_m)$}\label{distinctpoints} The point $c_m$ belongs to $\PP^1(F_m)$ and does not coincide with $c_i$ for $i<m$. 
		\borysitem{$(2_m)$}\label{Hilbert} 	For every $i\leqslant m$ the scheme $f_i^{-1}(c_i)$ is integral, and the function field $F_i(f_i^{-1}(c_i))$ of $f_i^{-1}(c_i)$ is linearly disjoint from $F_m/F_i$.
		\borysitem{$(3_m)$}\label{coveringorder}	The covering $f_m \in \calB$ is the first element of $\calB \setminus \{f_1,\dots,f_{m-1}\}$ defined over $F_m$.
		\borysitem{$(4_m)$}\label{fatorder}	The extension $L_m/K_m$ is the first element of $\calA \setminus \{L_{m-1}/K_{m-1},\dots,L_1/K_1\}$ such that $K_m$ is a subfield of $F_m$ and the extensions $F_m/K_m$ and $L_m/K_m$ are linearly disjoint.
		\borysitem{$(5_m)$}\label{nofatextensions}	The extension $F_{m}/F_{m-1}$ is a nontrivial finite subextension of  $L_{m-1}F_{m-1}/F_{m-1}$.
	\end{enumerate}	

	We do so inductively. Set $F_1=k$. Let $f_1 \in \calB$ be the first element defined over $k$. Let $L_1/K_1$ be the first element of $\calA$ with $K=k$; if no such element exists, then $F=k$ is a finite extension satisfying the conclusions of the theorem, which we assumed in the beginning to not exist. Choose a point $c_1 \in \PP^1(k)$ such that $f_1^{-1}(c_1)$ is integral; such a point exists because $k$ is Hilbertian. The properties \ref{distinctpoints}\ndash \ref{fatorder} are satisfied for $m=1$, while the property \ref{nofatextensions} for $m=1$ is vacuous.
	
	 Suppose a sequence $F_m, f_m, c_m, L_m/K_m$ is defined for $m<n$ and satisfies \ref{distinctpoints}\ndash \ref{nofatextensions}. We start by constructing the field $F_n$. Let $M$ denote the compositum of $F_{n-1}(f_i^{-1}(c_i))$ for all $i<n$; it is a finite extension of $F_{n-1}$. Let $L$ denote the compositum $F_{n-1}L_{n-1}$. By Lemma \ref{vastbasechange} and property \ref{fatorder} applied for $m=n-1$, the field extension $L/F_{n-1}$ is vast. Since $L/F_{n-1}$ is vast and $M$ is finite, we can choose a subextension $F_n$ of $L/F_{n-1}$ linearly disjoint from $M$. Let $f_n\in \calB$ be the element defined by condition \ref{coveringorder} for $m=n$. Let $L_n/K_n \in \calA$ be defined by \ref{fatorder} for $m=n$; if no such element exists then by property \ref{nofatextensions} for $m<n$, the finite extension $F_n/k$ satisfies the conclusion of the theorem, which does not happen by assumption. Since $F_n$ is a finite extension of a Hilbertian field it itself is Hilbertian. Therefore there exists a point $c_n \in \PP^1(F_n)$ distinct from $c_i$ for $i<n$ and such that $f_n^{-1}(c_n)$ is integral.  With these choices conditions \ref{distinctpoints}\ndash\ref{nofatextensions} are satisfied for all $m<n+1$. 
	
         Let $F$ denote the union $\bigcup_n F_n$. We claim that $F$ is Hilbertian. Indeed, suppose we are given a covering of $f\colon X \to \PP^1$ defined over $F$, with $X$ a smooth geometrically integral curve and $\deg f \geqslant 2$. Since a covering is defined by finitely many equations, $f$ will be defined over $F_n$ for some $n$. By condition \ref{coveringorder} and the definition of $\calB$ infinitely many of the coverings $f_m$ are equal to $f$. Let $c_{n_1}, c_{n_2}, \ldots \in \PP^1(F)$ be the corresponding sequence of points. By condition \ref{Hilbert} the schemes $f^{-1}(c_{n_i})$ are integral over $F$. Therefore $F$ is Hilbertian.
On the other hand, conditions \ref{fatorder} and \ref{nofatextensions} ensure that there is no vast extension $L/K$ in $\calA$ such that $L$ is linearly disjoint from $F$ over $K$.
Finally, by condition \ref{nofatextensions} the degree of $F$ divides the product of degrees of extensions from $\calA$.
\end{proof} 

Before proving Theorem \ref{Odonifail} we recall some group-theoretic properties of the profinite group $\Aut \mathrm{T}_\infty(d)$. Consider the action of $\Aut \mathrm{T}_\infty(d)$ on the $n$-th level of the tree. The sign of this action defines a homomorphism $\sigma_n\colon \Aut \mathrm{T}_\infty(d) \to \Z/2\Z$. Note that a simple transposition on the $n$-th level of the tree lifted arbitrarily to an element $g \in \Aut \mathrm{T}_\infty(d)$ satisfies $\sigma_m(g)=0$ for $m<n$ and $\sigma_n(g)=1$, therefore the homomorphisms $\sigma_n$ are linearly independent in the $\F_2$-vector space $\Hom(\Aut \mathrm{T}_\infty(d), \Z/2\Z)$.
Thus the collection of all sign homomorphisms defines a homomorphism $\sigma \colon \Aut \mathrm{T}_\infty(d) \to \prod_{n=1}^{\infty} \Z/2\Z$ whose image is dense in the product topology.
Since $\sigma$ is a continuous homomorphism of profinite groups, it is then surjective.
We use $\phi_n$ to denote the homomorphism $\phi_n\colonequals(\sigma_n, \sigma_{n+1}, \dots) \colon \Aut \mathrm{T}_\infty(d) \to \prod_{k=n}^\infty \Z/2\Z$.

\begin{definition}\label{discriminants}
  Suppose $K$ is a field and $f \in K[x]$ a monic polynomial of degree $d \geqslant 2$ such that $\phi_f \colon \Gal_K \to \Aut \mathrm{T}_\infty(d)$ has image of finite index.
  The $n$-th \slantsf{discriminant extension} $K_n$ of $K$ attached to $f$ is the algebraic extension of $K$ corresponding to the kernel of $\phi_n \circ \phi_f$ in $\Gal_K$.
\end{definition}

\begin{remark}
  The homomorphism $\phi_n\circ \phi_f$ of Definition \ref{discriminants} is not necessarily surjective.
  However, since the image of $\phi_f$ has finite index in $\Aut \mathrm{T}_\infty(d)$ by assumption, the image of $\phi_n\circ \phi_f$ has finite index in $\prod_{k=1}^{\infty} \Z/2\Z$, and so $\im \phi_n\circ \phi_f = \Gal K_n/K \simeq \prod_{i=1}^{\infty} \Z/2\Z$.
\end{remark}
\begin{remark}
	The term ``discriminant extension'' comes from the relation between the sign homomorphism and discriminants of polynomials; see \cite[Section 7.4.A]{Cox2011}.
\end{remark}

\begin{proof}[Proof of Theorem \ref{Odonifail}]
  Consider the set $S$ consisting of pairs $(K, f)$, where $K/k$ is a finite extension and $f \in K[X]$ is a monic polynomial of degree at least $2$ such that the image of $\phi_f \colon \Gal_K \to \Aut \mathrm{T}_\infty(\deg f)$ has finite index.
  Let
  \[ \calA \colonequals \{ K_n/K \colon (K, f) \in S,\ K_n \text{ is the $n$-th discriminant extension of $K$ attached to $f$ for some $n$} \}. \]
  Since every $\prod_{n=1}^{\infty} \Z/2\Z$-extension is vast, the collection $\calA$ satisfies the assumptions of Theorem \ref{MainTheorem}.
  Therefore there exists a Hilbertian extension $F/k$ of $2$-power degree such that $F$ is not linearly disjoint from any $K_n/K \in \calA$ with $K \subset F$.

  We wish to show that $F$ is as desired, so suppose for a contradiction that $f \in F[X]$ is a monic polynomial of degree at least $2$ with finite index arboreal representation over $F$.
  We can choose $K \subset F$ finite over $k$ with $f \in K[X]$, so that $(K, f) \in S$.
  The fields $K_n$ attached to $f$ (over $K$) are nested $K_1 \supset K_2 \supset \dots$ and $\bigcap_n K_n = K$.
  Since $F/K$ is not linearly disjoint from $K_n$ for any $n$, $F \cap K_1$ is an infinite extension of $K$.
  Therefore the arboreal representation of $f$ over $F$ has infinite index in $\Aut \mathrm{T}_\infty (\deg f)$, giving the desired contradiction.
\end{proof}


\section{A model-theoretic construction of a counterexample}
\label{sec:modelTheory}

The iterative construction from Theorem \ref{MainTheorem}, constructing an algebraic extension which step-by-step forces all polynomials to induce arboreal representations with image of infinite index, can also naturally be understood model-theoretically using the Omitting Types Theorem.
In this section we will therefore give another construction of counterexamples to Conjecture \ref{Odoni}, using standard tools from field arithmetic.

\begin{theorem}\label{thm:modelTheoreticFail}
  Let $K$ be a countable field of characteristic zero.
  There exists a Hilbertian pseudo-algebraically closed extension field $L$ of $K$ such that every monic polynomial over $L$ of degree $d \geqslant 2$ induces an arboreal representation whose image has infinite index in $\Aut \mathrm{T}_\infty(d)$.
  If $K$ is Hilbertian, we can choose $L$ to be algebraic over $K$.
\end{theorem}

\begin{remark}\label{rem:comparison}
  The fact that the resulting field $L$ is pseudo-algebraically closed should be seen as an artefact of the construction.
  In particular, this forces $L$ to have projective absolute Galois group (see \cite[Theorem 11.6.2]{Fried-Jarden2008}).
  If $K$ is a number field and $L/K$ is algebraic, this means that the degree of $L/K$ as a supernatural number must be divisible infinitely many times by every prime number for reasons of cohomological dimension (see \cite[Proposition 3.3.5]{Neukirch-Schmidt-Wingberg2008}), in contrast to the construction in the previous section.
\end{remark}

\begin{remark}\label{rem:firstProofImpliesSecond}
  Let $K$ be a countable field of characteristic zero.
  If $K$ is not Hilbertian, replace it by its Hilbertian extension $K(t)$.
  Then $K$ has an algebraic extension $K'$ which is Hilbertian and pseudo-algebraically closed (see \cite[Theorem 27.4.8]{Fried-Jarden2008}),
  and the proof in the previous section yields an algebraic extension $L/K'$ in which every monic polynomial of any degree $d \geqslant 2$ induces an arboreal representation whose image has infinite index in $\Aut \mathrm{T}_\infty(d)$.
  This field $L$ is then pseudo-algebraically closed as an algebraic extension of $K'$ \cite[Corollary 11.2.5]{Fried-Jarden2008}.
  Therefore the proof in the previous section yields Theorem \ref{thm:modelTheoreticFail} as a corollary.
  We nevertheless think that the separate proof below is interesting in its own right.
\end{remark}

We use basic model-theoretic terminology, with \cite{Hodges1997} as our general reference, although other textbooks such as \cite{Marker} also contain all necessary results.

We work in the first-order language of rings, i.e.~with symbols $+, -, \cdot, 0, 1$, later expanded by constants.
Let us introduce some terminology for sets $p(\underline x)$ consisting of formulae with free variables among the (finite) tuple of variables $\underline x$.
We say that a tuple $\underline a$ in a structure $\mathfrak M$ \emph{realises} $p$ if $\mathfrak M \models \varphi(\underline a)$ for all $\varphi \in p$.
Following \cite[Section 6.2]{Hodges1997}, we say that a formula $\varphi(\underline x)$ \emph{supports} the set $p(\underline x)$ in a theory $T$ if $T \cup \{ \exists \underline x \varphi \}$ has a model, and for every $\psi \in p$, $T \models \forall \underline x (\varphi \to \psi)$.\footnote{
  In \cite[Definition 4.2.1]{Marker} the terminology for this property is that $\varphi(\underline x)$ \emph{isolates} $p(\underline x)$, but there it is required throughout that $T \cup p(\underline x)$ be consistent (which is true in all interesting situations).
}
We say that the set $p(\underline x)$ is supported in $T$ if there exists a formula which supports it.

\begin{lemma}\label{lem:typeAlmostSurjRep}
  Fix $d \geqslant 2$, $n \geqslant 1$.
  There is a set $p_{d,n}(x_1, \dotsc, x_d)$ of formulae (in the language of rings) such that a tuple $(a_1, \dotsc, a_d)$ in a field $L$ realises $p_{d,n}$ if and only if the arboreal representation $\phi_f \colon \Gal_L \to \Aut \mathrm{T}_\infty(d)$ associated to $f = X^d + a_1X^{d-1} + \dotsb + a_d$ has image of index at most $n$ in $\Aut \mathrm{T}_\infty(d)$.
\end{lemma}
\begin{proof}
  For every $k \geqslant 1$, let us write $\mathrm{T}_k(d)$ for the part of $\mathrm{T}_\infty(d)$ up to level $k$.
  The image of $\phi_f$ has index at most $n$ if and only if the image of the finite stages $\Gal_L \to \Aut \mathrm{T}_k(d)$ has index at most $n$, for all $k$.
  Equivalently, the splitting field of $f^{\circ k}$ has degree at least $\lvert \Aut \mathrm{T}_k(d)\rvert / n$.
  Since the coefficients of $f^{\circ k}$ are polynomials in the $a_i$, this property is expressed by a first-order formula in the language of rings.
  Collecting these formulae for all $k$ yields the desired $p_{d,n}$.
\end{proof}

Let us now fix a countable field $K$ of characteristic zero, and work in the language $\mathcal{L}$ which is the language of rings together with constant symbols for each element of $K$.
Consider the $\mathcal{L}$-theory $T$ which consists of the theory of Hilbertian pseudo-algebraically closed fields (see \cite[Chapter 27]{Fried-Jarden2008}) and the diagram of $K$ (see \cite[Section 1.4]{Hodges1997}).
In this way, models of $T$ correspond to field extensions $L/K$ such that $L$ is Hilbertian and pseudo-algebraically closed.
For pseudo-algebraically closed fields, being Hilbertian is equivalent to the absolute Galois group satisfying a certain group-theoretic condition, called \emph{$\omega$-freeness}, see \cite[Theorem 5.10.3]{Jarden2011}.

Proving Theorem \ref{thm:modelTheoreticFail} now means showing that there exists a model $L$ of $T$ in which none of the sets $p_{d,n}$ from Lemma \ref{lem:typeAlmostSurjRep} is realised by any tuple.
(Let us ignore the additional condition that $L/K$ be algebraic if $K$ is Hilbertian for the moment.)
We will apply the following standard tool from model theory, which applies to any consistent theory $T$ in a countable language $\mathcal{L}$, thus in particular in our setting:

\begin{theorem}[{Omitting types, \cite[Theorem 6.2.1]{Hodges1997}; see also \cite[Theorem 4.2.4]{Marker}}]\label{thm:omittingTypes}
  For each $i \in \mathbb{N}$, let $p_i$ be a set of $\mathcal{L}$-formulae in free variables $x_1, \dotsc, x_{n_i}$, such that no $p_i$ is supported in $T$.
  Then there exists a model of $T$ realising none of the $p_i$, i.e.~the $p_i$ are \emph{omitted}.
\end{theorem}

We thus wish to show that none of the sets $p_{d,n}$ from Lemma \ref{lem:typeAlmostSurjRep} are supported in $T$.
We first isolate two important properties of $T$ from \cite[Theorem 27.2.3]{Fried-Jarden2008}, rephrasing the model-theoretic terminology there in more elementary terms.
\begin{lemma}\label{lem:modelCompleteness}
  Let $\varphi(\underline x)$ be an $\mathcal{L}$-formula.
  Then $\varphi$ is equivalent modulo $T$ to a formula of the form $\exists \underline y (\psi(\underline x, \underline y))$, where $\psi$ is a positive boolean combination of polynomial equalities and formulae of the form $\forall z (z^m + t_1z^{m-1} + \dotsb + t_m \neq 0)$, where the $t_i$ are polynomial expressions in $\underline x, \underline y$ with coefficients from $K$.
\end{lemma}
\begin{proof}
  This is a relatively straightforward translation of the model completeness part of \cite[Theorem 27.2.3]{Fried-Jarden2008}.
  Let us explain this in detail.

  The theorem cited works in an extended language $\mathcal{L'}$ for fields, containing not only symbols $+, -$ and $\cdot$, but also an additional $n$-ary predicate $R_n$ for every $n \geqslant 2$, and states in particular that the theory $T'$ consisting of $T$ and the additional axioms
  \[ \forall x_1, \dotsc, x_n \big( R_n(x_1, \dotsc, x_n) \leftrightarrow \exists z (z^n + x_1z^{n-1} + \dotsb + x_n = 0)\big) \]
  for all $n$ (specifying the intended interpretation for $R_n$) is \emph{model complete}.

  By \cite[Theorem 7.3.1]{Hodges1997}, this implies that the formula $\neg\varphi(\underline x)$ is equivalent modulo $T'$ to a universal $\mathcal{L'}$-formula $\theta_0(\underline x)$, i.e.\ a formula built by universal quantification, $\wedge$ and $\vee$ from atomic formulae and their negations.
  Here atomic formulae are equations between terms (i.e. polynomial expressions in the variables and constants), as well as formulae of the form $R_n(t_1, \dotsc, t_n)$ with terms $t_i$.

  Taking negations, $\varphi(\underline x)$ is equivalent modulo $T'$ to $\neg\theta_0(\underline x)$, which in turn is equivalent to an existential formula $\theta_1(\underline x)$, i.e.\ one built from atomic formulae and their negations using existential quantification, $\wedge$ and $\vee$.
  We may eliminate all negated equalities $t_1 \neq t_2$ between terms by rewriting them as $\exists z (zt_1 = zt_2 + 1)$.
  Now replace all occurrences of $\neg R_n(t_1, \dotsc, t_n)$ in $\psi_1$ by $\forall z(z^n + x_1z^{n-1} + \dotsb + x_n \neq 0)$ and all occurrences of $R_n(t_1, \dotsc, t_n)$ by $\exists z(z^n + x_1z^{n-1} + \dotsb + x_n = 0)$ to obtain an $\mathcal{L}$-formula $\theta_2(\underline x)$.
  Moving all existential quantifiers to the front of the formula, $\theta_2$ is equivalent to an $\mathcal{L}$-formula $\theta_3 = \exists \underline y(\psi(\underline x, \underline y))$ with $\psi$ of the required shape.

  We have seen that $\varphi$ is equivalent to $\theta_3$ modulo $T'$.
  Since $T'$ is a so-called definitional expansion of $T$, i.e.\ only adds axioms specifying the interpretation of additional relation symbols, $\varphi$ and $\theta_3$ are actually already equivalent modulo $T$, since any model of $T$ in which they are not equivalent could be expanded to a model of $T'$ \cite[Theorem 2.6.4(a)]{Hodges1997}.
\end{proof}

\begin{lemma}\label{lem:existsRegularExtension}
  For every extension $L/K$ there exists a regular extension $F/L$ with $F \models T$.
\end{lemma}
\begin{proof}
  This is a translation of another part of \cite[Theorem 27.2.3]{Fried-Jarden2008}, as we now explain.
  By the statement about model companions given there, $L$ embeds into a model $F$ of $T$ such that the embedding $L \hookrightarrow F$ not only preserves addition and multiplication, but embeds $L$ as a relatively algebraically closed subfield of $F$, see \cite[top of p.~660]{Fried-Jarden2008}.
  Identifying the embedding with an inclusion, this means that $F/L$ is regular since the characteristic is zero.
\end{proof}

We now prove the key technical statements for the proof of Theorem \ref{thm:modelTheoreticFail}.
\begin{lemma}\label{lem:almostSurjNotSupported}
  For any fixed $d \geqslant 2$ and $n$, the set $p_{d,n}$ from Lemma \ref{lem:typeAlmostSurjRep} is not supported in $T$.
\end{lemma}
\begin{proof}
  Let $\varphi(\underline x)$ be an arbitrary $\mathcal{L}$-formula such that $T \cup \{ \exists \underline x(\varphi(\underline x)) \}$ is consistent; we show that $\varphi$ does not support $p_{d,n}$ in $T$.
  We may assume that $\varphi = \exists \underline y(\psi(\underline x, \underline y))$ as in Lemma \ref{lem:modelCompleteness}.
  Since $T \cup  \{ \exists \underline x \varphi \}$ is consistent, there is a Hilbertian pseudo-algebraically closed field $L/K$ with elements $\underline x, \underline y$ such that $L \models \psi(\underline x, \underline y)$.
  We can choose an algebraic extension $L'/L$, with $\Gal_{L'}$ finitely generated, such that $L' \models \psi(\underline x, \underline y)$, since we only need to ensure that the finitely many polynomials $Z^m + t_1 Z^{m-1} + \dotsb + t_m$ mentioned in $\psi$ which have no root in $L$ do not have any root in $L'$.
  Then the arboreal representation $\phi_f \colon \Gal_{L'} \to \Aut \mathrm{T}_\infty(d)$ associated to $f = X^d + x_1X^{d-1} + \dotsb + x_m$ has image of infinite index, since $\Aut \mathrm{T}_\infty(d)$ is not finitely generated because of the surjective homomorphism $\Aut \mathrm{T}_\infty(d) \to \prod_{k=1}^\infty \Z/2\Z$.

  By Lemma \ref{lem:existsRegularExtension}, there exists a regular extension $F/L'$ with $F \models T$.
  Since the representation $\Gal_F \to \Aut \mathrm{T}_\infty(d)$ associated to $f$ factors through $\Gal_{L'}$, its image has infinite index, so the tuple $\underline x$ does not realise $p_{d,n}$ in $F$.
  As $F \models \varphi(\underline x)$ by construction, $\varphi$ does not support $p_{d,n}$.
\end{proof}

\begin{lemma}\label{lem:nonAlgNotSupported}
  Assume in addition that the fixed countable field $K$ is Hilbertian.
  Let $p(x)$ be the set of formulae in one variable which assert that $x$ satisfies no nontrivial polynomial relation over $K$.
  Then $p$ is not supported in $T$.
\end{lemma}
\begin{proof}
  Suppose $\varphi(x)$ is a formula supporting $p$.
  In particular, there exists an extension $L/K$ with $L \models T \cup \{ \exists x \varphi(x) \}$.
  By the downward Löwenheim-Skolem Theorem, we may assume that $L$ is countable.
  By \cite[Proposition 23.2.4]{Fried-Jarden2008},
  $L$ is elementarily equivalent over $K$ to an ultraproduct $L' = \prod_{n=1}^\infty L_n/\mathcal{D}$, where $\mathcal{D}$ is an ultrafilter on $\mathbb{N}$ and each $L_n$ is an algebraic extension of $K$ which is PAC and satisfies $\Gal_{L_n} \cong \Gal_L$, so in particular every $L_n$ is $\omega$-free.
  Since $L' \equiv L \models \exists x \varphi(x)$, by \L{}oś's theorem on ultraproducts \cite[Corollary 7.7.2]{Fried-Jarden2008} there exists an $n$ such that $L_n \models \exists x \varphi(x)$.
  Now $L_n$ contains no element realising $p$, but has an element satisfying $\varphi$, so we deduce that $\varphi$ does not support $p$.
\end{proof}

\begin{proof}[Proof of Theorem \ref{thm:modelTheoreticFail}]
	The theorem follows immediately from Theorem \ref{thm:omittingTypes} with Lemma \ref{lem:almostSurjNotSupported}, by simultaneously omitting the sets $p_{n,d}$ for all $n$ and $d$, and additionally omitting the set $p$ from Lemma \ref{lem:nonAlgNotSupported} if $K$ is Hilbertian.
\end{proof}


\section*{Acknowledgements} 

We thank Padmavathi Srinivasan for useful comments that improved the exposition. We thank the anonymous referees for many helpful suggestions.

\end{document}